\documentclass[a4,11pt]{article}
\title{{\sc Number of Eigenvalues of Non-self-adjoint Schr\"{o}dinger Operators with Dilation Analytic Complex Potentials}}
\author{Norihiro Someyama
\footnote{Department of Mathematics, Gakushuin University, 1-5-1 Mejiro, Toshima-ku, Tokyo 171-8588, Japan / e-mail: philomatics@outlook.jp}}
\date{\empty}

\usepackage{latexsym,amsmath,amssymb,amsthm,amscd,mathrsfs}
\usepackage{bm}
\usepackage[dvipdfmx]{graphicx}
\usepackage{fancybox}
\usepackage{multicol}
\usepackage{tabularx}

\newtheorem{thm}{Theorem}[section]
\newtheorem{prop}{Proposition}[section]
\newtheorem{cor}{Corollary}[section]
\newtheorem{lem}{Lemma}[section]

\theoremstyle{definition}

\newtheorem{assump}{Assumption}[section]
\newtheorem{rem}{Remark}[section]
\newtheorem{ex}{Example\!\!}

\theoremstyle{theorem}

\makeatletter
 
 \@addtoreset{equation}{section}
\makeatother

\begin{document}
\maketitle

\begin{abstract}
In the present paper, we give Lieb-Thirring type inequalities for all isolated eigenvalues of $d$-dimensional non-self-adjoint Schr\"{o}dinger operators with complex-valued and dilation analytic potentials.
In order to derive them, we prove that isolated eigenvalues and their multiplicities are invariant under complex dilation.
\end{abstract}
\vspace{3mm}

{\small 
{\bf Keywords}:
non-self-adjoint Schr\"{o}dinger operator, Lieb-Thirring inequality, complex potential, dilation analytic potential, isolated eigenvalue
}

\section{Introduction}
In the present paper, we consider Schr\"{o}dinger operators on ${\mathbb R}^d$,
\[
H:=H_0+V,\quad \ H_0:=-\Delta=-\sum_{j=1}^d\frac{\partial^2}{\partial x_j^2}
\]
with complex potentials $V$ which are dilation analytic (see Assumption \ref{assump:CDA} for definition), where the domain ${\cal D}(H_0)$ of $H_0$ is the second order Sobolev space in ${\mathbb R}^d$, $H^2({\mathbb R}^d)$.
We denote the real and the imaginary parts of $z\in \mathbb C$ by $\Re z$ and $\Im z$ respectively.
We prove that the sum of moments of isolated eigenvalues of $H$ in the half-plane can be estimated by the Lieb-Thirring type inequalities
\begin{equation}
\label{eq:Sineq}
\sum_{\lambda\in \sigma_{{\rm d}}(H)\,\cap\,[{\mathbb C}^{\pm}\cup\,(-\infty,0)]}|\lambda|^{\gamma}
\le C_{\gamma,d}\int_{{\mathbb R}^d}|V_{\pm \pi {\rm i}/4}(x)|^{\gamma+d/2}\,{\rm d}x
\end{equation}
for $d,\gamma\ge 1$, where ${\mathbb C}^{\pm}:=\{z\in \mathbb C:\pm \Im z>0\}$ and $V_{\theta}(x):=V({\rm e}^{\theta}x)$.
Here and hereafter, the formulas which contain $\pm$ represent two formulas, one for the upper sign and the other for the lower sign. 
The space of $L^p$-functions from the space $X$ to the space $Y$ will be denoted by $L^p(X;Y)$, in particular we denote $L^p(X;X)=L^p(X)$.

In addition, many researches on the number of eigenvalues of Schr\"{o}dinger operators with complex potentials have been done.
For instance, Frank, Laptev and Safronov(\cite{FLS}) recently derived the following result for the number $N$ of isolated eigenvalues of the {\it odd} dimensional Schr\"{o}dinger operator $H$ with the complex potential which decays exponentially fast:
\[
N\le \frac{C_d}{\varepsilon^2}\left(\int_{\mathbb R^d}{\rm e}^{\varepsilon |x|}|V(x)|^{(d+1)/2}\,{\rm d}x\right)^2
\]
for any $\varepsilon>0$, where $C_d$ is a constant depending only on $d=3,5,7,\ldots$.
We also give an inequality for the number of {\it all} isolated eigenvalues of {\it any} dimensional Schr\"{o}dinger operator $H$ with the special dilation analytic complex potential as a corollary of (\ref{eq:Sineq}).

\subsection{Lieb-Thirring Inequalities for Real Potentials}
We recall the standard Lieb-Thirring inequality.
We consider the self-adjoint Schr\"{o}dinger operator $H=H_0+V$ in $L^2({\mathbb R}^d)$ defined by the closure of the quadratic form 
\[
q_H(u):=\int_{{\mathbb R}^d}|\nabla u(x)|^2\,{\rm d}x+\int_{{\mathbb R}^d}V(x)|u(x)|^2\,{\rm d}x,\quad u\in C_0^{\infty}({\mathbb R}^d)
\]
where $\nabla:=(\partial_{x_j})_{1\le j\le d}$ with the derivative $\partial_{x_j}$ in the sense of distibutions. 
If $V\in L^{\gamma+d/2}({\mathbb R}^d;\mathbb R)$ and
\begin{equation}
\lambda_1\le \lambda_2\le \cdots \le \lambda_n\le \cdots<0
\label{eq:neglam}
\end{equation}
are negative eigenvalues of $H$, {\it the standard Lieb-Thirring inequality}:
\begin{equation}
\sum_{n=1}^{\infty}|\lambda_n|^{\gamma}\le L_{\gamma,d}\int_{{\mathbb R}^d}V_{-}(x)^{\gamma+d/2}\,{\rm d}x,\quad \ V_{\pm}:=\max\{\pm V,0\}
\label{eq:sLT}
\end{equation}
is well known(e.g., \cite{AL}, \cite{LW1}, \cite{LW2}, \cite{LL}, \cite{LS} and \cite{LT}), here $L_{\gamma,d}$ is the {\it sharp constant} and $\gamma$ satisfies
\begin{equation}
\left.
\begin{array}{ll}
\displaystyle \gamma \ge \frac{1}{2},&\ d=1\vspace{1.5mm}\\
\gamma>0,&\ d=2\vspace{1mm}\\
\gamma\ge 0,&\ d\ge 3
\end{array}
\right\}.
\label{eq:gammad}
\end{equation}
It is important that $L_{\gamma,d}$ do not depend on $V$.
If $\gamma=0$, the left-hand side of (\ref{eq:sLT}) is the number of negative eigenvalues (\ref{eq:neglam}) of $H$.
For the constant $L_{\gamma,d}$, it is well known that 
\[
L_{\gamma,d}^{{\rm cl}}\le L_{\gamma,d}
\]
for all $\gamma\ge 0$ and $d\ge 1$, where $L_{\gamma,d}^{{\rm cl}}$ is the classical constant:
\[
L_{\gamma,d}^{{\rm cl}}:=(2\pi)^{-d}\int_{{\mathbb R}^d}(|\xi|-1)_{-}^{\gamma}\,{\rm d}\xi=\frac{\Gamma(\gamma+1)}{2^d\pi^{d/2}\Gamma(\gamma+d/2+1)}.
\]
In fact, it is proven that 
\begin{equation}
L_{\gamma,d}=L_{\gamma,d}^{{\rm cl}}
\label{eq:LLcl}
\end{equation}
for all $d\ge 1$ if $\gamma \ge 3/2$ (cf. \cite{LT},\cite{AL},\cite{LW2}), and 
\begin{equation}
L_{\gamma,d}\le \frac{\pi}{\sqrt{3}}L_{\gamma,d}^{{\rm cl}}=(1.814\cdots)\times L_{\gamma,d}^{{\rm cl}}
\label{eq:iLc}
\end{equation}
for all $d,\gamma \ge 1$ (\cite{DLL}).
Frank, Hundertmark, Jex and Nam(\cite{FHJN}) however improved (\ref{eq:iLc}) to
\[
L_{\gamma,d}\le 1.456\times L_{\gamma,d}^{{\rm cl}}
\]
if $\gamma=1$.
On the other hand, Helffer and Robert(\cite{HR}) proved that 
\[
L_{\gamma,d}^{{\rm cl}}<L_{\gamma,d}
\]
if 
\[
\left.
\begin{array}{ll}
\displaystyle \frac{1}{2}\le \gamma<\frac{3}{2}, & d=1 \vspace{1mm}\\
\gamma<1, & d\ge 2
\end{array}
\right\}.
\]
For further information about $L_{\gamma,d}$ and $L_{\gamma,d}^{\rm cl}$, see Remarks of Theorem 12.4 of \cite{LL}.

Lieb-Thirring inequalities are key ingredients in the proof of the stability of matter by Lieb and Thirring(e.g., \cite{LS}), and is used for obtaining the efficient lower bound for the energies of fermions.

\subsection{Lieb-Thirring Inequalities for Complex Potentials}
The Lieb-Thirring inequality (\ref{eq:sLT}) has recently been extended to Schr\"{o}dinger operators with complex potentials $V\in L^{\gamma+d/2}({\mathbb R}^d;\mathbb C)$ for $d,\gamma\ge 1$.

We suppose that $V$ is a complex-valued potential which is $H_0$-compact.
Then
\[
H=H_0+V,\quad {\cal D}(H)={\cal D}(H_0)
\]
is quasi-maximal accretive(\cite{K}), and the spectrum $\sigma(H)$ of $H$ consists of the essential spectrum $\sigma_{{\rm ess}}(H)=[0,\infty)$ and the discrete spectrum $\sigma_{{\rm d}}(H)$ which consists of isolated eigenvalues of $H$ of finite algebraic multiplicities:
\[
m_{\lambda}:=\sup_{k\in \mathbb N}\left(\dim \ker[(H-\lambda)^k]\right).
\]
This can be seen via analytic Fredholm theory(\cite{RS1}) applied to 
\[
(H-z)^{-1}=(H_0-z)^{-1}(1+V(H_0-z)^{-1})^{-1},
\]
because
\begin{itemize}
\item $F(z):=V(H_0-z)^{-1}$ is an analytic function of $z\in \mathbb C\setminus [0,\infty)$ with values in the space of compact operators, and
\item $(1-F(z))^{-1}$ exists for $z\in \mathbb C\setminus [0,\infty)$ with sufficiently large $|\Im z|$.
\end{itemize}
Since $H$ is non-self-adjoint, $m_{\lambda}$ is in general different from 
the geometric multiplicity defined by
\[
g_{\lambda}:=\dim \{u\in L^2({\mathbb R}^d):(H-\lambda)u=0\},
\]
and we count the number of eigenvalues according to their algebraic 
multiplicities. 

Then, the following estimates for the sum of moments of eigenvalues of $H$ {\it outside the cone ${\cal C}_{\kappa}:=\{z\in \mathbb C:|\Im z|<\kappa \Re z\}$} for {\it any positive} constant $\kappa$ are proven:

\begin{thm}[Frank, Laptev, Lieb and Seiringer; \cite{FLLS}]
\label{thm:FLLS}
Let $d\ge 1$ and $\gamma\ge 1$. 
Suppose that $V\in L^{\gamma+d/2}({\mathbb R}^d;{\mathbb C})$.
Then, for any $\kappa>0$, 
\begin{equation}
\sum_{\sigma_{{\rm d}}(H)\,\cap\,{\cal C}_{\kappa}^{\rm c}}|\lambda|^{\gamma}\le C_{\gamma,d}\left(1+\frac{2}{\kappa}\right)^{\gamma+d/2}\int_{{\mathbb R}^d}|V(x)|^{\gamma+d/2}\,{\rm d}x
\label{eq:FLLS}
\end{equation}
where ${\cal C}_{\kappa}^{\rm c}$ is the complement set of ${\cal C}_{\kappa}$ and 
\begin{equation}
C_{\gamma,d}:=2^{1+\gamma/2+d/4}L_{\gamma,d}
\label{eq:Cgammad}
\end{equation}
and $L_{\gamma,d}$ is the constant of (\ref{eq:sLT}). In particular, 
\begin{equation}
\sum_{\sigma_{{\rm d}}(H)\,\cap\,\{\Re z\le 0\}}|\lambda|^{\gamma}
\le C_{\gamma,d}\int_{{\mathbb R}^d}|V(x)|^{\gamma+d/2}\,{\rm d}x.
\label{eq:momegainfi}
\end{equation}
\end{thm}

Frank et al.(\cite{FLLS}) conjecture that (\ref{eq:FLLS}) and (\ref{eq:momegainfi}) in Theorem \ref{thm:FLLS} hold for $\gamma$ satisfying (\ref{eq:gammad}). 
On the other hand, observing that (\ref{eq:sLT}) is equivalent to
\begin{equation}
\sum_{n=1}^{\infty}\frac{{\rm dist}(\lambda_n;[0,\infty))^{\gamma+d/2}}{|\lambda_n|^{d/2}}\le L_{\gamma,d}\int_{{\mathbb R}^d}V_{-}(x)^{\gamma+d/2}\,{\rm d}x
\label{eq:dis-sLT}
\end{equation}
where ${\rm dist}(x;\Omega)$ is a distance from a point $x$ to the domain $\Omega$ and $\{\lambda_n\}$ are negative eigenvalues of $H$ given by (\ref{eq:neglam}), Demuth, Hansmann and Katriel(\cite{DHK1}) proposed to study the estimate: 
\begin{equation}
\sum_{\lambda \in \sigma_{{\rm d}}(H)}\frac{{\rm dist}(\lambda;[0,\infty))^{\gamma+d/2}}{|\lambda|^{d/2}}\le C_{\gamma,d}\int_{{\mathbb R}^d}|V(x)|^{\gamma+d/2}\,{\rm d}x
\label{eq:DHKconj}
\end{equation}
for the constant $C_{\gamma,d}$ independent of $V\in L^{\gamma+d/2}({\mathbb R}^d;\mathbb C)$ and for $\gamma$ and $d$ satisfying (\ref{eq:gammad}).
They(\cite{DHK3}) actually proved by applying (\ref{eq:FLLS}) that for any $\gamma \ge 1$ and $0<\tau<1$,
\begin{equation}
\sum_{\lambda \in \sigma_{{\rm d}}(H)}\frac{{\rm dist}(\lambda;[0,\infty))^{\gamma+d/2+\tau}}{|\lambda|^{d/2+\tau}}\le C_{\gamma,d,\tau}\int_{{\mathbb R}^d}|V(x)|^{\gamma+d/2}\,{\rm d}x
\label{eq:DHKineq}
\end{equation}
with constant
\[
C_{\gamma,d,\tau}=\frac{{\rm const.}}{\tau}.
\]
Note that this constant {\it blows up} as $\tau\downarrow 0$ (see also \cite{DHK2} where a similar estimate is obtained).

Related to this problem, Frank and Sabin(\cite{FS}) proved that if $d\ge 1$ and $V\in L^p({\mathbb R}^d;\mathbb C)$ such that
\[
\left\{
\begin{array}{ll} 
p=1,&\ d=1,\vspace{1mm}\\
\displaystyle 1<p\le \frac{3}{2},&\ d=2, \vspace{1mm}\\
\displaystyle \frac{d}{2}<p\le \frac{d+1}{2},&\ d\ge 3,
\end{array}
\right.
\]
then 
\[
\sum_{\lambda \in \sigma_{{\rm d}}(H)}\frac{{\rm dist}(\lambda;[0,\infty))}{|\lambda|^{(1-\varepsilon)/2}}\le C_{p,d,\varepsilon}\left(\int_{{\mathbb R}^d}|V(x)|^{p}\,{\rm d}x\right)^{(1+\varepsilon)/(2p-d)}
\]
where $\varepsilon$ is the non-negative number fulfilling:
\[
\left.
\begin{array}{ll}
\varepsilon>1, & d=1 \vspace{1mm}\\
\varepsilon\ge 0, & \displaystyle d\ge 2\ {\rm and}\ \frac{d}{2}\le p\le \frac{d^2}{2d-1} \vspace{1mm}\\
\displaystyle \varepsilon>\frac{(2d-1)p-d^2}{d-p}, & \displaystyle d\ge 2\ {\rm and}\ \frac{d^2}{2d-1}\le p\le \frac{d+1}{2}
\end{array}
\right\}.
\]

In recent years, Cuenin, Laptev, Safronov etc. studied the eigenvalues of $H$ which are close to $[0,\infty)$.
For example, they(\cite{LS}) proved that if $\Re V\ge 0$ is bounded and $\Im V\in L^p({\mathbb R}^d)$ and
\[
\left.
\begin{array}{ll}
p\ge 1, & d=1 \vspace{1mm}\\
\displaystyle p>\frac{d}{2}, & d\ge 2
\end{array}
\right\},
\]
then one has
\[
\sum_{\lambda \in \sigma_{{\rm d}}(H)}\left(\frac{\Im \lambda}{|\lambda+1|^2+1}\right)_+^p\le C_{p,d}\int_{{\mathbb R}^d}(\Im V)_+(x)^p\,{\rm d}x
\]
where
\[
C_{p,d}:=(2\pi)^{-d}\int_{{\mathbb R}^d}\frac{{\rm d}\xi}{(|\xi|^2+1)^p}.
\]

\section{Assumption and Main Results \label{sec:sec2}}
We write 
$(\cdot,\cdot)$ for the $L^2({\mathbb R}^d)$-inner product and $\|\cdot\|$ for the $L^2({\mathbb R}^d)$-norm:
\[
(f,g):=\int_{{\mathbb R}^d}f(x)\overline{g(x)}\,{\rm d}x,\qquad \|f\|:=(f,f)^{1/2}.
\]

\subsection{Dilation Analytic Method for Complex Potentials}
We consider the 1-parameter unitary group $U(\theta)$, 
$\theta \in \mathbb R$, on $L^2({\mathbb R}^d)$ defined by 
\begin{equation}
U(\theta)u(x):={\rm e}^{d\theta/2}u({\rm e}^{\theta}x),\quad u\in L^2({\mathbb R}^d).
\end{equation}
We now suppose that $V$ fulfills the following assumption.
Recall that ${\cal D}(T)$ is the domain of the operator $T$.
Moreover, we write ${\bf B}(S_1,S_2)$ for the space of bounded operators from the space $S_1$ to the space $S_2$, in particular ${\bf B}(S)$ is ${\bf B}(S,S)$ if $S_1=S_2=S$.

\begin{assump}
\label{assump:CDA}
Let $d,\gamma \ge 1$.
We assume the followings from beginning to end.
\begin{enumerate}
\item[a)] $V$ is the multiplication operator with the complex-valued measurable function satisfying $V\in L^{\gamma+d/2}({\mathbb R}^d;\mathbb C)$.
\item[b)] The operator $V$ is $H_0$-compact, that is, ${\cal D}(V)\supset {\cal D}(H_0)=H^2({\mathbb R}^d)$ and $V(H_0+1)^{-1}$ is compact.
\item[c)] The function $V_{\theta}(x):=V({\rm e}^{\theta}x)$ originally defined for $\theta\in \mathbb R$ has an analytic continuation 
into the complex strip
\[
\mathscr{S}_{\alpha}:=\{z\in \mathbb C:|\Im z|<\alpha\}\quad 
{\rm for}\ {\rm some}\ \alpha>0
\]
as an $L^{\gamma+ d/2}(\mathbb R^d;\mathbb C)$-valued function. 
\item[d)] The function $V_\theta(H_0+1)^{-1}$ originally defined for $\theta\in {\mathbb R}$ can be extended to $\mathscr{S}_\alpha$ as a ${\bf B}(L^2({\mathbb R}^d))$-valued analytic function. 
\end{enumerate}
\end{assump}

In what follows, we fix $d,\gamma \ge 1$. 
We call $V$ fulfilling Assumption \ref{assump:CDA} {\it the dilation analytic complex potential on $\mathscr{S}_{\alpha}$}.
We define
\begin{equation}
\left\{
\begin{array}{l}
H_0(\theta):={\rm e}^{-2\theta}H_0, \vspace{2mm}\\ 
H(\theta):=U(\theta)HU(\theta)^{-1}=H_0(\theta)+V_{\theta}
={\rm e}^{-2\theta}\left(H_0+{\rm e}^{2\theta}V_{\theta}\right)
\end{array}
\right.
\label{eq:Hthetae-2theta}
\end{equation}
for $\theta\in \mathscr{S}_{\alpha}$. 
It is obvious that 
$H_0(\theta)$ and $H(\theta)$ are operator-valued holomorphic 
functions of type (A) of $\theta\in \mathscr{S}_{\alpha}$ in the 
sense of Kato (\cite{K}). 
Moreover
\begin{equation}
H(\theta+\theta')=U(\theta')H(\theta)U(\theta')^{-1}; \quad 
\theta\in \mathscr{S}_{\alpha}, \ \theta'\in \mathbb R, 
\label{eq:HUHU-1}
\end{equation}
since this is true for $\theta\in \mathbb{R}$ and 
both sides of (\ref{eq:HUHU-1}) are ${\bf B}(H^2({\mathbb R}^d);L^2({\mathbb R}^d))$-valued analytic functions of $\theta \in \mathscr{S}_{\alpha}$. 
In particular, $\sigma(H(\theta))$ is independent of $\Re \theta$. 

The following result for real-valued potentials is a well known fact.
We pay attention to that the same result holds for complex-valued potentials.

\begin{prop}
If $V$ is a dilation analytic complex potential on $\mathscr{S}_{\alpha}$, one has
\begin{equation}
\sigma_{{\rm ess}}(H(\theta))={\rm e}^{-2\theta}[0,\infty)
=\{{\rm e}^{-2\theta}x:x\in [0,\infty)\}
\label{eq:sigessHtheta}
\end{equation}
for any $\theta\in \mathscr{S}_{\alpha}$.
\end{prop}

\begin{proof}
Define, for fixed $\theta \in \mathscr{S}_{\alpha}$,
\begin{equation}
\tilde{H}(\theta):=H_0+{\rm e}^{2\theta}V_{\theta}={\rm e}^{2\theta}H(\theta).
\label{eq:Htilde}
\end{equation}
That $\sigma(\tilde{H}(\theta))\setminus [0,\infty)$ consists of 
isolated eigenvalues of finite multliplicities may be proved as previously by the argument using analytic Fredholm theory for 
\[
F(z):=-{\rm e}^{2\theta}V_{\theta}(H_0-z)^{-1}.
\] 
In  particular, $\sigma_{{\rm ess}}(\tilde{H}(\theta))\subset [0,\infty)$. 
We show that $[0,\infty)\subset \sigma_{{\rm ess}}(\tilde{H}(\theta))$.
Suppose that there is an open interval $(a,b)\subset [0,\infty)$ such that 
\[
(a,b)\subset \mathbb C\setminus \sigma(\tilde{H}(\theta)).
\] 
Then, $\sigma(H_0)\cap (a,b)$ must be a discrete set by virtue of the argument above where the roles of $\tilde{H}(\theta)$ and $H_0$ are replaced.
This is of course impossible and $[0,\infty)\subset \sigma_{{\rm ess}}(\tilde{H}(\theta))$.
Therefore, 
\begin{equation}
\sigma_{{\rm ess}}(\tilde{H}(\theta))=\sigma_{{\rm ess}}(H_0)=[0,\infty), 
\label{eq:tHthetaH0}
\end{equation}
and
\[
\sigma_{{\rm ess}}(H(\theta))={\rm e}^{-2\theta}\sigma_{{\rm ess}}(\tilde{H}(\theta))={\rm e}^{-2\theta}[0,\infty)
\]
by (\ref{eq:Htilde}) and (\ref{eq:tHthetaH0}).
\end{proof}

\begin{rem}
(\ref{eq:sigessHtheta}) is immediately derived by Weyl's theorem for essential spectrums if $H$ is self-adjoint. (See e.g. Theorem XIII.36 of \cite{RS4}.)
\end{rem}

\subsection{Estimates on Eigenvalues in $\mathbb C\setminus [0,\infty)$}
The main result in the present paper is the following theorem.

\begin{thm}
\label{thm:EEonC}  
Let $V$ be a dilation analytic complex potential on $\mathscr{S}_{\alpha}$ with $\alpha>\pi/4$. 
Suppose that $C_{\gamma,d}$ is the constant given by (\ref{eq:Cgammad}).
Then, one has
\begin{equation}
\sum_{\lambda \in \sigma_{{\rm d}}(H)\,\cap\,[{\mathbb C}^{\pm}\cup (-\infty,0)]}|\lambda|^{\gamma}
\le C_{\gamma,d}\int_{{\mathbb R}^d}|V_{\pm \pi {\rm i}/4}(x)|^{\gamma+d/2}\,{\rm d}x.
\label{eq:EEonC1}
\end{equation}
\end{thm}

As follows, Theorem \ref{thm:EEonC} derives the Lieb-Thirring type inequality for {\it all} isolated eigenvalues of the non-self-adjoint Schr\"{o}dinger operator with the dilation analytic complex potential.

\begin{cor}
\label{cor:ULEE}
If $V$ is a dilation analytic complex potential on $\mathscr{S}_{\alpha}$ with $\alpha>\pi/4$, then one has
\[
\sum_{\lambda \in \sigma_{{\rm d}}(H)}|\lambda|^{\gamma}
\le C_{\gamma,d}\int_{{\mathbb R}^d}\left(|V_{\pi {\rm i}/4}(x)|^{\gamma+d/2}+|V_{-\pi {\rm i}/4}(x)|^{\gamma+d/2}\right)\,{\rm d}x
\]
where $C_{\gamma,d}$ is the constant of (\ref{eq:Cgammad}).
\end{cor}

\begin{rem}
It is predicted that in general the integrals on the right-hand side of (\ref{eq:EEonC1}) are bigger than $\int|V(x)|^{\gamma+d/2}\,{\rm d}x$.
(See the following example.)
However, we can of course take minimum constants $C^{\pm}$ satisfying
\[
\int_{{\mathbb R}^d}|V_{\pm \pi {\rm i}/4}(x)|^{\gamma+d/2}\,{\rm d}x\le C^{\pm}\int_{{\mathbb R}^d}|V(x)|^{\gamma+d/2}\,{\rm d}x
\]
respectively, so we come to the conclusion that {\it the Lieb-Thirring inequality (with a worse constant) for all isolated eigenvalues of ${\rm H}$ with the dilation analytic complex potential establishes} from Corollary \ref{cor:ULEE}:
\[
\sum_{\lambda \in \sigma_{{\rm d}}(H)}|\lambda|^{\gamma}
\le C\int_{{\mathbb R}^d}|V(x)|^{\gamma+d/2}\,{\rm d}x,\quad C:=2\max\{C_{\gamma,d}C^{-},C_{\gamma,d}C^{+}\}.
\]
Here, in general $C^{\pm}$ probably depend on $V$ in both cases.
\end{rem}

{\small
\begin{ex}
We consider a 1-dimensional complex potential
\begin{equation}
V(x)=\frac{c}{(1+x^2)^s},\quad x\in {\mathbb R}
\label{eq:Som-ex}
\end{equation}
where $c\in \mathbb C$ and $1/2<s<1$. 
$V$ is a dilation analytic complex potential on $\mathscr{S}_{\alpha}$ with $\alpha\in (\pi/4,\pi/2)$.
In fact, a singularity appears if $\theta=\pi {\rm i}/2$ since 
\[
V_{\theta}(x)=\frac{c}{(1+{\rm e}^{2\theta}x^2)^s},
\]
moreover $V_{\theta}\in L^{\gamma+1/2}({\mathbb R}^d)$ if and only if
\[
\frac{1}{2\gamma+1}<s<\frac{2}{2\gamma+1}.
\]
In particular $\gamma=1$, we obtain
\begin{align*}
\int_{-\infty}^{\infty}|V_{\pm \pi {\rm i}/4}(x)|^{3/2}\,{\rm d}x
&=\int_{-\infty}^{\infty}\frac{|c|^{3/2}}{\sqrt{1+x^4}^{3s/2}}\,{\rm d}x \\
&\ge \int_{-\infty}^{\infty}\frac{|c|^{3/2}}{(1+x^2)^{3s/2}}\,{\rm d}x \\
&=\int_{-\infty}^{\infty}|V(x)|^{3/2}\,{\rm d}x
\end{align*}
as expected.
By the way, if $V:{\mathbb R}^d\to \mathbb C$ is a complex-valued potential like (\ref{eq:Som-ex}), that is, if $V$ satisfies that 
\[
|V(x)|\le \frac{{\rm const.}}{(1+|x|^2)^s},\quad \ \frac{1}{2}<s<1,
\]
it is known(\cite{Sa}) that all non-real eigenvalues of the non-self-adjoint Schr\"{o}dinger operator $H=-\Delta+V$ are in a disc of a finite radius.
\end{ex}
}

We denote the number of the isolated eigenvalues of the operator $T$ in a domain $\Omega$ by $N(T;\Omega)$.
If we add another condition besides Assumption \ref{assump:CDA} as follows, we can mention the number of all isolated eigenvalues of $H$.

\begin{cor}
Let $H$ be a dissipative Schr\"{o}dinger operator with the dilation analytic complex potential on $\mathscr{S}_{\alpha}$ with $\alpha>\pi/4$.
That is, we suppose that
\begin{itemize}
\item[i)] the function $\mathbb R^d\ni x\mapsto V(x)$ satisfies $\Im V(x)<0$ for all $x\in \mathbb R^d$ and \item[ii)] there exists some constant $C$ such that
\begin{equation}
\label{eq:Vdecay}
|\Re V(x)|,|\Im V(x)|\le C\langle x\rangle^{-\rho}
\end{equation}
outside the sufficiently large sphere, where $\langle x\rangle:=(1+|x|^2)^{1/2}$ and $\rho>1$.
\end{itemize}
Then, one has
\begin{equation}
\label{eq:NHC}
N(H;\mathbb C\setminus[0,\infty))\le \tilde{C}_{\gamma,d}\int_{{\mathbb R}^d}\left(|V_{\pi {\rm i}/4}(x)|^{\gamma+d/2}+|V_{-\pi {\rm i}/4}(x)|^{\gamma+d/2}\right)\,{\rm d}x
\end{equation}
where $\tilde{C}_{\gamma,d}:=C_{\gamma,d}/\inf_{\lambda \in \sigma_{{\rm d}}(H)}|\lambda|^{\gamma}$.
\label{cor:EEonC}
\end{cor}

\begin{proof}
We put $\Lambda:=\inf_{\lambda\in \sigma_{{\rm d}}(H)}|\lambda|^{\gamma}$ for simplicity.
First, 
\[
\sum_{\lambda\in \sigma_{{\rm d}}(H)}|\lambda|^{\gamma}
\ge \left(\inf_{\lambda\in \sigma_{{\rm d}}(H)}|\lambda|^{\gamma}\right)\sum_{\lambda\in \sigma_{{\rm d}}(H)}1=\Lambda N(H;\mathbb C\setminus [0,\infty)).
\]
Next, it is proved that $0$ is not an accumulation point of isolated eigenvalues of $H$ if $H$ is dissipative and $V$ satisfies (\ref{eq:Vdecay}) (See Theorem 1.1 of \cite{W}), so $\Lambda>0$.
Hence, (\ref{eq:NHC}) holds from the above.
\end{proof}

\section{Proof of Theorem\,\ref{thm:EEonC}\label{sec:PT}}
The estimate (\ref{eq:momegainfi}) of Theorem \ref{thm:FLLS} plays an important role in the proof of the main theorem.
The points of a proof of Theorem \ref{thm:EEonC} are that isolated eigenvalues of $H$ and their multiplicities are invariant under complex dilation.
Let us prove them by dividing into two lemmmas.

We begin with the following lemma. 
Recall that $\sigma(H(\theta))$ is discrete in 
$\mathbb{C}\setminus \sigma_{{\rm ess}}(H(\theta))$ and all $\lambda\in \sigma_{{\rm d}}(H(\theta))$ have finite algebraic multiplicities. 

\begin{lem}
Suppose that $0<\alpha<\pi/2$ and $\theta \in \mathscr{S}_\alpha \cap {\mathbb C}^{\pm}$.
Then, for $\lambda\in \mathbb{C}^{\pm}$,  
$\lambda \in \sigma_{\rm d}(H)$ if and only if $\lambda \in \sigma_{{\rm d}}(H(\theta))$.
\label{lem:Som1}
\end{lem}

\begin{proof} 
We prove only the case $\theta \in \mathscr{S}_\alpha \cap {\mathbb C}^{+}$ 
and $\lambda \in {\mathbb C}^+$.  
The other case may be proved simliarly. 
Suppose that $\lambda(\theta)\in \mathbb{C}^{+}$ is an eigenvalue of $H(\theta)$ for $\theta\in \mathscr{S}_{\alpha}\cap \mathbb{C}^{+}$.
We show that if $\lambda$ is an eigenvalue of $H(\theta_0)$ for
$\theta_0 \in \mathscr{S}_{\alpha}\cap \mathbb{C}^{+}$ then $\lambda$ is an eigenvalue 
of $H(\theta)$ for all $\theta \in \mathscr{S}_{\alpha}\cap \mathbb{C}^{+}$. To see this, 
we recall that $\{H(\theta)\}_{\theta\in \mathscr{S}_{\alpha}\cap \,\mathbb{C}^{+}}$ is an analytic family of type (A) in the sense of Kato. Then, Theorem 1.8 of Chapter VII \S1.3 (or Chapter II \S1) of \cite{K} implies that those $\lambda(\theta)\in \sigma_{{\rm d}}(H(\theta))$ such that $\lambda(\theta)\to \lambda$ as $\theta\to \theta_0$ are given by the branches of one or several analytic functions as Puiseux series.
Moreover, as remarked above, $\lambda(\theta)$ is independent of $\Re \theta$. 
It follows that $\lambda(\theta)=\lambda$ for all $\theta$ near $\theta_0$, and $\lambda(\theta)=\lambda,\ \theta\in \mathscr{S}_{\alpha}\cap {\mathbb C}^{+}$.
\end{proof}

\begin{lem}
Let $V$ be a dilation analytic complex potential on $\mathscr{S}_{\alpha}$ with $\alpha>\pi/4$.
Suppose that $\lambda \in {\mathbb C}^{\pm}$ is an eigenvalue of $H$ 
and $\theta \in \mathscr{S}_{\alpha}\cap \mathbb C^{\pm}$. 
Then, the algebraic multiplicity of $\lambda$ as eigenvalue 
of $H(\theta)$ coincides with that of $\lambda$ as eigenvalue of $H$.
\label{lem:Som2}
\end{lem}

\begin{proof}
We define two Riesz projections onto the generalized eigenspaces:
\begin{equation}
P_{\lambda}:=\frac{-1}{2\pi {\rm i}}\oint_{\Gamma_{\varepsilon}}(H-z)^{-1}\,{\rm d}z,\quad \ 
P_{\lambda}(\theta):=\frac{-1}{2\pi {\rm i}}\oint_{\Gamma_{\varepsilon}}(H(\theta)-z)^{-1}\,{\rm d}z
\label{eq:Plam}
\end{equation}
for $\theta\in \mathscr{S}_{\alpha}\cap \mathbb C^{\pm}$, where $\Gamma_{\varepsilon}$ is the contour around $\lambda$ defined by
\begin{equation}
\Gamma_{\varepsilon}:=\{z\in \mathbb C:|z-\lambda|=\varepsilon\}
\label{eq:contour}
\end{equation}
with sufficient small $\varepsilon>0$ such that $\Gamma_{\varepsilon}$ does not enclose any other point in $\sigma(H)$ or $\sigma(H(\theta))$ except $\lambda$. 
Here, notice that $P_{\lambda}$ is not necessary an orthogonal projection, since $H$ is non-self-adjoint.
It is well known that $P_{\lambda},P_{\lambda}(\theta)$ are projections and 
\[
P_{\lambda}^2=P_{\lambda},\quad \ 
P_{\lambda}(\theta)^2=P_{\lambda}(\theta)
\]
onto the respective generalized eigenspaces of $H$ and $H(\theta)$ associated with the eigenvalue $\lambda$.
It follows that
\[
m_{\lambda}={\rm rank}\,P_{\lambda},\quad \ m_{\lambda(\theta)}={\rm rank}\,P_{\lambda}(\theta).
\]
We write $P_{\lambda}'(\theta)$ for the orthogonal projection onto $M_{\lambda}(\theta):=P_{\lambda}(\theta)L^2({\mathbb R}^d)$.
In order to prove this lemma, we show that the dimension of $M_{\lambda}(\theta)$ is independent of $\theta$.
$P_{\lambda}(\theta)$ is analytic with respect to $\theta$.
It is known that
\[
\|P_{\lambda}'(\theta)-P_{\lambda}'(\sigma)\|\le \|P_{\lambda}(\theta)-P_{\lambda}(\sigma)\|
\]
for any $\theta,\sigma\in \mathscr{S}_{\alpha}$ (see Theorem 6.35 in Chapter I $\S$6.8 of \cite{K}).
Since $P_{\lambda}(\theta)\to P_{\lambda}(\sigma)$ as $\theta\to \sigma$ in norm by the analyticity of $P_{\lambda}(\theta)$ (see (\ref{eq:Plam})) and 
\[
\left(P_{\lambda}'(\theta)-P_{\lambda}'(\sigma)\right)^2+\left(1-P_{\lambda}'(\theta)-P_{\lambda}'(\sigma)\right)^2=1,
\]
one has
\[
\|P_{\lambda}'(\theta)-P_{\lambda}'(\sigma)\|<1.
\]
Thus, $P_{\lambda}'(\theta)$ and $P_{\lambda}'(\sigma)$ are unitarily equivalent by Theorem 6.32 in Chapter I $\S$6.8 of \cite{K}.
Hence
\[
\dim M_{\lambda}(\theta)=\dim M_{\lambda}(\sigma),
\]
so the proof of this lemma ends by putting $\sigma=0$.
\end{proof}

We denote $\mathbb R_{\pm}:=\{\pm x:x>0\}$ and ${\rm i}\,{\mathbb R}_{\pm}:=\{{\rm i}\,x:x\in \mathbb R_{\pm}\}$.
\\

\noindent
{\it Proof of Theorem \ref{thm:EEonC}.}
For eigenvalues $\tilde{\lambda}(\theta)$ of $\tilde{H}(\theta)={\rm e}^{2\theta}H(\theta)$, we write 
\begin{equation}
\lambda(\theta)={\rm e}^{-2\theta}\tilde{\lambda}(\theta)
\label{eq:lamtil}
\end{equation}
for the corresponding eigenvalue of $H(\theta)$, $\theta \in \mathscr{S}_{\alpha}$.

We first set $\theta=\pi{\rm i}/4$.
Lemma \ref{lem:Som1}, Lemma \ref{lem:Som2} and (\ref{eq:lamtil}) imply that 
\begin{equation}
{\rm e}^{\pi{\rm i}/2}\left(\sigma_{{\rm d}}(H)\cap (\mathbb C^{+}\cup \mathbb R_-)\right)=\sigma_{{\rm d}}(\tilde{H}(\pi{\rm i}/4))\cap \bigl(\{\Re z<0\}\cup {\rm i}\mathbb R_-\bigr)
\label{eq:ilam}
\end{equation}
including their multiplicities.
We next set $\theta=-\pi{\rm i}/4$.
We likewise have
\begin{equation}
{\rm e}^{-\pi{\rm i}/2}\left(\sigma_{{\rm d}}(H)\cap (\mathbb C^{-}\cup \mathbb R_-)\right)=\sigma_{{\rm d}}(\tilde{H}(-\pi{\rm i}/4))\cap \bigl(\{\Re z<0\}\cup {\rm i}\mathbb R_+\bigr)
\label{eq:-ilam}
\end{equation}
including their multiplicities.
We write $\mathfrak{S}_{\mp}(\tilde{H}(\pm \pi{\rm i}/4))$ for the right-hand side of (\ref{eq:ilam}) and (\ref{eq:-ilam}) respectively, and apply the estimate (\ref{eq:momegainfi}) of Theorem\,\ref{thm:EEonC} to $\tilde{H}(\pm \pi{\rm i}/4))=H_0\pm {\rm i}\,V_{\pi{\rm i}/4}$.
This implies
\begin{align*}
\sum_{\lambda \in \sigma_{{\rm d}}(H)\,\cap\,({\mathbb C}^{\pm}\cup \,\mathbb R_-)}|\lambda|^{\gamma}&=\sum_{\tilde{\lambda}(\pm \pi{\rm i}/4)\in \mathfrak{S}_{\mp}(\tilde{H}(\pm \pi{\rm i}/4))}|\tilde{\lambda}(\pm \pi{\rm i}/4)|^{\gamma} \nonumber\\
&\le C_{\gamma,d}\int_{{\mathbb R}^d}|V_{\pm \pi {\rm i}/4}(x)|^{\gamma+d/2}\,{\rm d}x.\nonumber
\end{align*}
Since $m_{\tilde{\lambda}(\pi{\rm i}/4)}=m_{\tilde{\lambda}(-\pi{\rm i}/4)}=m_{\lambda}$ for any $\lambda\in \sigma_{{\rm d}}(H)\cap (\mathbb C^{\pm}\cup \mathbb R_-)$ by virtue of Lemma \ref{lem:Som2}, this completes the proof.
\qed

\vspace{4mm}

{\small 
\begin{center}
{\bf {\sc Acknowledgement}}
\end{center}
The author heartily thank professor Kenji Yajima for advice and instruction.
Moreover, the author would like to thank the referees of ROMP for precious advice and Dr. Ryszard Mruga\l a for guidance respectively.
}
\end{document}